\documentclass[12pt]{article}
\usepackage{graphics,color}
\usepackage{tikz}

\usepackage{amsmath, amsthm, amssymb}
\usepackage{graphicx}
\usepackage{enumerate}

\DeclareMathOperator{\Circ}{Circ}
\DeclareMathOperator{\m}{mod}

\newtheorem{theorem}{Theorem}
\newtheorem{lemma}[theorem]{Lemma}
\newtheorem{proposition}[theorem]{Proposition}
\newtheorem{corollary}[theorem]{Corollary}
\newtheorem{conjecture}[theorem]{Conjecture}

\theoremstyle{remark}

\theoremstyle{remark}
\newtheorem{example}{Example}
\theoremstyle{definition}
\newcommand{\f}[2]{\displaystyle \frac{#1}{#2}}

\date{June 5, 2013}

\newlength{\cfk}

\title{Finite Prime Distance Graphs and 2-Odd Graphs}

\author{Joshua D. Laison, Colin Starr, and Andrea Walker \\ Mathematics Department, Willamette University \\ 900 State St., Salem, OR 97301}

\begin{document}
\maketitle

\begin{abstract}
A graph $G$ is a prime distance graph (respectively, a 2-odd graph) if its vertices can be labeled with distinct integers such that for any two adjacent vertices, the difference of their labels is prime (either 2 or odd).  We prove that trees, cycles, and bipartite graphs are prime distance graphs, and that Dutch windmill graphs and paper mill graphs are prime distance graphs if and only if the Twin Prime Conjecture and dePolignac's Conjecture are true, respectively.  We give a characterization of 2-odd graphs in terms of edge colorings, and we use this characterization to determine which circulant graphs of the form $\Circ(n, \{1,k\})$ are 2-odd and to prove results on circulant prime distance graphs.
\end{abstract}



\section{Introduction}

Prime distance graphs were introduced by Eggleton, Erd{\H{o}}s, and Skilton in 1985 \cite{Eggleton85,Eggleton86}.  For any set $D$ of positive integers, they defined the \textit{distance graph} $\mathbb{Z}(D)$ as the graph with vertex set $\mathbb{Z}$ and an edge between integers $x$ and $y$ if and only if $|x-y| \in D$.  The \textit{prime distance graph} $\mathbb{Z}(P)$ is the distance graph with $D=P$, the set of all primes.  They proved that the chromatic number $\chi(\mathbb{Z}(P))=4$.  Research in prime distance graphs has since focused on the chromatic number of $\mathbb{Z}(D)$ where $D$ is a non-empty proper subset of $P$ \cite{Eggleton88,Eggleton90,Voigt94,Yegnanarayanan02}.  Note that these graphs are all infinite (non-induced) subgraphs of $\mathbb{Z}(P)$.  In this paper we investigate finite subgraphs of $\mathbb{Z}(P)$.

Specifically, we say that a graph $G$ is a \textit{\textbf{prime distance graph}} if there exists a one-to-one labeling of its vertices $L:V(G) \to \mathbb{Z}$ such that for any two adjacent vertices $u$ and $v$, the integer $|L(u)-L(v)|$ is prime.  We also define $L(uv)=|L(u)-L(v)|$.  We call $L$ a \textit{\textbf{prime distance labeling}} of $G$, so $G$ is a prime distance graph if and only if there exists a prime distance labeling of $G$.  We sometimes denote a vertex with label $i$ by $(i)$.  Note that in a prime distance labeling, the labels on the vertices of $G$ must be distinct, but the labels on the edges need not be.  Also note that by our definition, $L(uv)$ may still be prime if $uv$ is not an edge of $G$.

We say that $G$ is \textit{\textbf{2-odd}} if for any two adjacent vertices $u$ and $v,$ $|L(u)-L(v)|$ is either odd or exactly 2, in which case $L$ is a \textit{\textbf{2-odd labeling}} of $G$.  Note that prime distance graphs are trivially 2-odd.  Corollary~\ref{odd-not-prime} below shows that not every 2-odd graph is a prime distance graph.
\bigskip

\begin{example} The path $P_n$ is 2-odd and prime-distance for each $n$.  Figure~\ref{path} shows a prime distance (and therefore 2-odd) labeling of $P_n$.
\end{example} \hfill $\square$ \vspace{.15in}
\begin{figure}[h!]
\setlength{\unitlength}{1in}
\begin{center}
\begin{tikzpicture}
\draw (0,0)node[anchor=north]{0}--(1,0)node[anchor=north]{3}--(2,0)node[anchor=north]{6}--(3,0)node[anchor=north]{9};
\begin{scope}[dashed]
\draw (3,0)--(4,0)node[anchor=north]{$3(n-1)$};
\end{scope}
\draw (4,0) -- (5,0)node[anchor=north]{$3n$};
\draw[fill=black](0,0) circle(1mm);
\draw[fill=black](1,0) circle(1mm);
\draw[fill=black](2,0) circle(1mm);
\draw[fill=black](3,0) circle(1mm);
\draw[fill=black](4,0) circle(1mm);
\draw[fill=black](5,0) circle(1mm);
\end{tikzpicture} \caption{A prime distance labeling of $P_n$.} \label{path}
\end{center}
\end{figure}
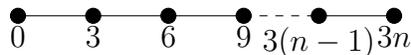

\section{Connections to Well-known Statements in Number Theory}

Surprisingly, the existence of prime distance labelings of some infinite families of graphs is closely related to several well-known statements in Number Theory.  In this section we show that all bipartite graphs are prime distance graphs using the Green-Tao Theorem; we give three separate proofs that all cycle graphs are prime distance graphs using the Goldbach Conjecture, Vinogradov's Theorem, and Ramar\'{e}'s Theorem, respectively; and we show that Dutch windmill graphs and paper mill graphs are prime distance graphs if and only if the Twin Prime Conjecture and dePolignac's Conjecture are true, respectively.

Recall that an \textit{\textbf{arithmetic progression}} of positive integers is a sequence of positive integers such that the differences between successive terms of the sequence is a constant.
\bigskip

\noindent \textbf{The Green-Tao Theorem.} \cite{green08} \textit{For any positive integer $k$, there exists a prime arithmetic progression of length $k$.}
\bigskip

\begin{theorem} Every bipartite graph is a prime distance graph. \label{bipartite} \end{theorem}

\begin{proof} We show that the graph $K_{r,s}$ is a prime distance graph.  Since every subgraph of a prime distance graph is also a prime distance graph, this proves that every bipartite graph is a prime distance graph.  By the Green-Tao Theorem, there is an arithmetic sequence of $r+s-1$ primes $p-(r-1)k$, $p-(r-2)k$, $\ldots$, $p-k$, $p$, $p+k$, $\ldots$, $p+(s-2)k$, $p+(s-1)k$.  Let $A$ and $B$ be the partite sets of $G$ with $|A|=r$ and $|B|=s.$  Label the members of $B$ with the labels $p$, $p+k$, $\ldots$, $p+(s-1)k$ and the members of $A$ with the labels $0,k,2k,\ldots, (r-1)k.$  Then differences between members of $A$ and members of $B$ are all of the form $p+nk$, with $n\in\{-(r-1)$, $-(r-2)$, $\ldots$, -1, 0, 1, $\ldots$, $s-2$, $s-1\},$ and each such $p+nk$ is prime. \end{proof}

In other words, every 2-chromatic graph is a prime distance graph.  However, we note that not every 3-chromatic graph is a prime distance graph since $K_{3,3,3}$ is not a prime-distance graph.  In fact, $K_{3,3,3}$ is not 2-odd.
\bigskip

\noindent \textbf{Goldbach's Conjecture.} \cite{Burton11,Rosen10} \textit{Every even number greater than $2$ is the sum of two primes.}
\bigskip

\begin{theorem}
  If Goldbach's Conjecture is true, then every cycle is a prime distance graph.
\end{theorem}

\begin{proof}
If $n=3$, then $C_n$ can be prime distance labeled with labels $0$, $3$, and $5$.  If $n=4$ then $C_n$ can be prime distance labeled with labels $0$, $3$, $8$, and $11$.  If $n=5,$ then $C_n$ can be prime distance labeled with labels $0, 3, 6, 9,$ and $11$.  Suppose $n \geq 6$, and write $2n-4$ as the sum of two primes, say $2n-4=p_1+p_2$.  Then $C_n$ can be prime distance labeled with labels $0$, $2$, $\ldots$, $2n-4$, and $p_1$ in cyclic order.  Since $2n-4$ is even and at least $6$, then $p_1$ must be odd, so the vertex labels are distinct.
\end{proof}

\noindent \textbf{Vinogradov's Theorem.} \cite{Ramachandra97,Vinogradov04} \textit{Every sufficiently large odd number is the sum of 3 primes.}
\bigskip

\begin{theorem} Every cycle is a prime distance graph. \label{cycle-theorem}
\end{theorem}

\begin{proof}
If $n=3$, then $C_n$ can be prime distance labeled with labels $0$, $3$, and $5$.  Suppose $n \geq 4$, and let $p$ be a prime number large enough such that by Vinogradov's Theorem, $p+2n-8$ can be written as the sum of 3 primes, say $p+2n-8=p_1+p_2+p_3$.  Also assume $p>4n$ and $p_1 \geq p_2 \geq p_3$, so in particular $p_1>2n-8$.  Then $C_n$ can be prime distance labeled with labels $0$, $2$, $\ldots$, $2n-8$, $p+2n-8$, $p_1+p_2$, and $p_1$ in cyclic order.  Since $p_1>2n-8$, the vertex labels are distinct.
\end{proof}
\bigskip

\noindent \textbf{Ramar\'{e}'s Theorem.} \cite{Ramare95} \textit{Every even number is the sum of at most $6$ primes.}
\bigskip

\begin{theorem}
Every cycle is a prime distance graph.
\end{theorem}

\begin{proof}
If $3 \leq n \leq 7$ then $C_n$ can be prime distance labeled using any of the techniques above.  Suppose $n \geq 8$, let $p$ be a prime number larger than $10n$, and write $2n-5+p$ as the sum of at most $6$ primes, $p_1$ through $p_i$, where $2 \leq i \leq 6$.  Again assume $p_1$ is the largest of these primes, so $p_1>2n-5$.  Then we have $5$ cases:
\bigskip

\noindent \textbf{Case 1. $2n-5+p=p_1+p_2$.} Then $C_n$ can be prime distance labeled with labels $0$, $2$, $4, \ldots, 2n-8, 2n-5, 2n-5+p$, and $p_1$.
\bigskip

\noindent \textbf{Case 2. $2n-5+p=p_1+p_2+p_3$.} Then $C_n$ can be prime distance labeled with labels $0, 2, 4, \ldots, 2n-10, 2n-5$, $2n-5+p$, $p_1+p_2$, and $p_1$.
\bigskip

\noindent \textbf{Case 3. $2n-5+p=p_1+p_2+p_3+p_4$.} Then $C_n$ can be prime distance labeled with labels $0, 2, 4, \ldots, 2n-12, 2n-5$, $2n-5+p$, $p_1+p_2+p_3$, $p_1+p_2$, and $p_1$.
\bigskip

\noindent \textbf{Case 4. $2n-5+p=p_1+p_2+p_3+p_4+p_5$.} Then $C_n$ can be prime distance labeled with labels $0, 2, 4, \ldots, 2n-18, 2n-15, 2n-10, 2n-5$, $2n-5+p$, $p_1+p_2+p_3+p_4$, $p_1+p_2+p_3$, $p_1+p_2$, and $p_1$.
\bigskip

\noindent \textbf{Case 5. $2n-5+p=p_1+p_2+p_3+p_4+p_5+p_6$.} Then $C_n$ can be prime distance labeled with labels $0, 2, 4, \ldots, 2n-20, 2n-15, 2n-10, 2n-5$, $2n-5+p$, $p_1+p_2+p_3+p_4+p_5$, $p_1+p_2+p_3+p_4$, $p_1+p_2+p_3$, $p_1+p_2$, and $p_1$.

In each case, since $p_1>2n-5$, the vertex labels are distinct.
\end{proof}

Recently, an announcement was made of a proof of the Weak Goldbach Conjecture \cite{Helfgott13}.  The authors believe this theorem could also be used to prove that every cycle is a prime distance graph.

The \textit{\textbf{Dutch windmill graph}} $D_n$ or \textit{\textbf{friendship graph}} is the star $S_{2n}$ with central vertex $v_0$ and leaves $v_1$ through $v_{2n}$, with an edge between each consecutive pair of vertices $v_{2k-1}$ and $v_{2k}$, $1 \leq k \leq n$.  So $D_n$ has $n$ copies of $C_3$ joined at the common vertex $v_0$ \cite{Gallian98,Erdos66}.  Figure~\ref{windmill} shows $D_5$.
\begin{figure}
\begin{center}
\includegraphics{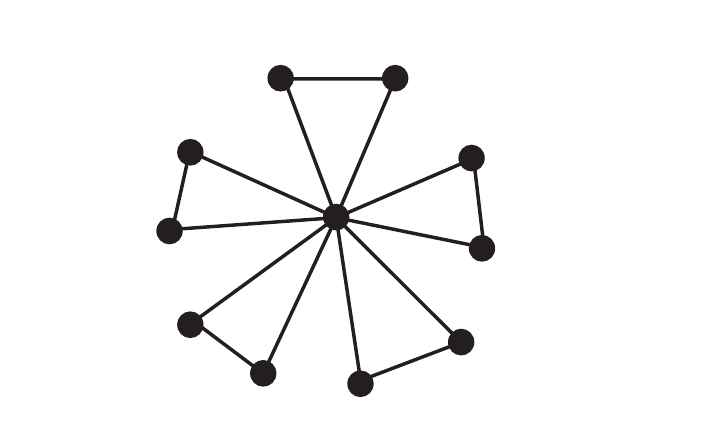}
\caption{The Dutch windmill graph $D_5$.}\label{windmill}
\end{center}
\end{figure}
\bigskip

\noindent \textbf{The Twin Prime Conjecture.} There are infinitely many pairs of primes that differ by 2.
\bigskip

\begin{theorem}
Every Dutch windmill graph is a prime distance graph if and only if the Twin Prime Conjecture is true.
\end{theorem}

\begin{proof}
First assume that $D_n$ has a prime distance labeling for any positive integer $n$, and consider one such prime distance labeling of $D_n$.  Without loss of generality, assume that $v_0$ is labeled with 0.  Note that there must be at most two more even labels on the remaining vertices, so at least $n-2$ of the triangles in $D_n$ have two odd labels.  In each of these triangles, since both odd-labeled vertices are adjacent to $v_0$, their labels must be prime.  Since their difference is even and prime, it must be 2, so each triangle is labeled with a pair of twin primes.  That is, if $D_n$ is prime distance, there are at least $n-2$ twin primes.  Therefore, if all Dutch windmill graphs are prime distance graphs, then the Twin Prime conjecture is true.

Conversely, if the Twin Prime Conjecture is true, then the labeling with $0$ on $v_0$ and a pair of twin primes on each triangle is a prime distance labeling of $D_n$.
\end{proof}

We construct the paper mill graphs as follows.  First, the \textit{\textbf{(triangular) book graph}} $B_5$ is the tripartite graph $K_{1,1,5}$ consisting of $5$ triangles sharing a common edge, the \textit{\textbf{spine}} of the book \cite{West96}.  The \textit{\textbf{stack of books}} $S_k$ is a union of $k$ copies of $B_5$ joined so that their spines form a path, as shown in Figure~\ref{book}.  Then let the \textit{\textbf{paper mill graph}} $M_{n,k}$ be the graph constructed from $D_n$ by replacing each edge $uv$ not incident to the center vertex $v_0$ by a copy of $S_k$ at vertices $a$ and $b$.
\bigskip

\begin{figure}
\begin{center}
\includegraphics[width=5in]{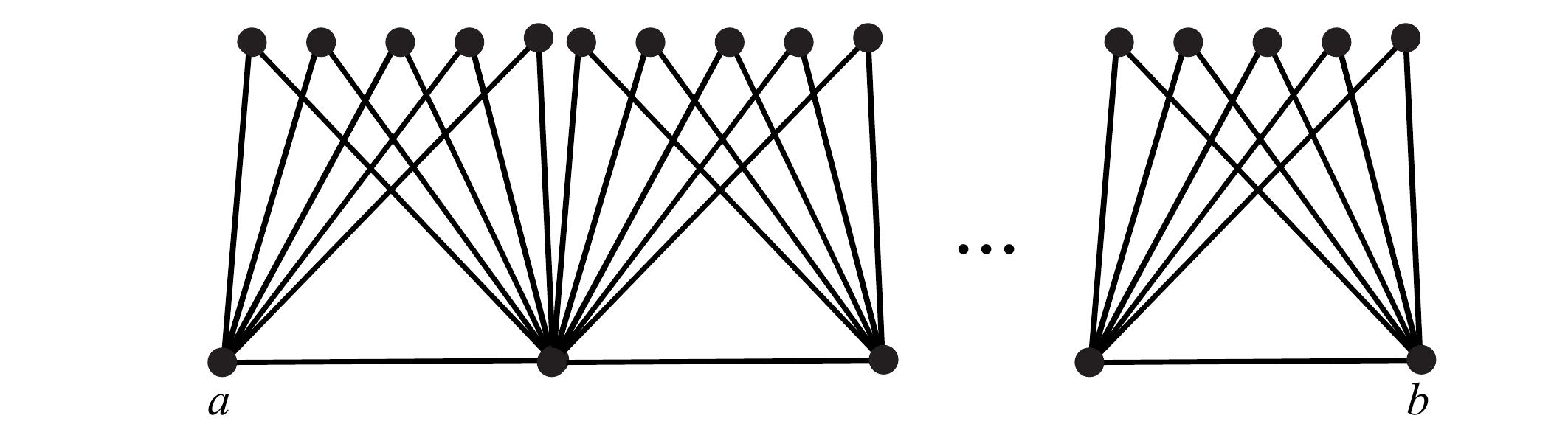}
\caption{The stack of books $S_k$.}\label{book}
\end{center}
\end{figure}
\bigskip

\noindent \textbf{dePolignac's Conjecture.} \cite{Dickson66,polignac} \textit{For any positive even integer $2k$, there exist infinitely many pairs of consecutive primes that differ by $2k$.}
\bigskip

\begin{lemma}
In any prime distance labeling of $S_k$, the labels on the vertices $a$ and $b$ differ by exactly $2k$. \label{stack-lemma}
\end{lemma}

\begin{proof}
By Corollary~\ref{odd-cycle} below, in any prime distance labeling of $B_5$, each 3-cycle must have exactly one edge labeled $2$.  But at most two of these edges can be incident to any single vertex, so the spine of $B_5$ must be labeled $2$.  Therefore all the edges in the path between $a$ and $b$ are labeled $2$.
\end{proof}

\begin{theorem}
Every paper mill graph is a prime distance graph if and only if dePolignac's Conjecture is true.
\end{theorem}

\begin{proof}
First assume that $M_{n,k}$ has a prime distance labeling for any positive integers $n$ and $k$, and consider one such prime distance labeling of $M_{n,k}$.  Without loss of generality, assume that the center vertex $v_0$ is labeled with 0.  Then all the vertices adjacent to $v_0$ are labeled with primes.  By Lemma~\ref{stack-lemma}, the ends of the paths on each $S_k$ make up $n$ pairs of primes such that each pair differs by $2k$.  Therefore, if all paper mill graphs are prime distance graphs, dePolignac's Conjecture is true.

Conversely, assume that dePolignac's Conjecture is true; note that this also implies the Twin Prime Conjecture.  We construct a prime distance labeling of $M_{n,k}$ for given positive integers $n$ and $k$.  Label $v_0$ with $0$ and the $i$th path of spines by $p_i$, $p_i+2$, $\ldots$, $p_i+2k$, where $p_i$ and $p_i+2k$ are prime.  We choose each $p_i$ so that $p_i$ is sufficiently large that labels do not repeat.  Then label the five additional vertices in the $j$th book with distinct integers so that the five pairs of edges are twin primes; do this for each book along $S_i$ without repeating labels.  (This is possible by assumption.)  Repeat this for each $S_i.$
\end{proof}

\section{Edge Colorings}

Given a prime distance (respectively, 2-odd) graph, we color red the edges labeled with 2 and blue the edges labeled with odd primes (odd integers).  Conversely, given a graph $G$ with edges colored red and blue (an \textit{\textbf{edge-colored}} graph), a prime distance (respectively, 2-odd) labeling of $G$ is \textit{\textbf{color-satisfying}} if the label on every red edge is 2 and the label on every blue edge is prime (odd).  We say that an edge-coloring of $G$ is a \textit{\textbf{prime distance coloring}} (respectively, \textit{\textbf{$2$-odd coloring}}) if there exists a color-satisfying prime distance labeling ($2$-odd labeling) of $G$.

Suppose that $G$ is edge-colored, and for two vertices $u$ and $v$ of $G$, the edge $uv$ is red.  If $L$ is a color-satisfying prime distance (respectively, 2-odd) labeling of $G$, then $L(uv)=2$, so we may assume without loss of generality that
$L(u)=0$ and $L(v)=2$.

The following proposition shows that not every color-satisfying 2-odd labeling will lead to a prime-distance labeling.  Throughout this paper, red and blue edges will be represented by thick and thin edges, respectively.
\bigskip

\begin{proposition} The graph $G$ shown in Figure \ref{2-odd-a} is a prime distance graph but has a 2-odd coloring that is not a prime-distance coloring.
\end{proposition}

\begin{proof}
First note that by Theorem~\ref{n3-theorem} below, $G$ is a prime distance graph.  Now the first labeling of $G$ shown in Figure~\ref{2-odd-a} is a color-satisfying 2-odd labeling.  Note that $G$ has two red paths, each with $5$ vertices.  If there were a prime-distance labeling $L$ of $G$, without loss of generality, one of these paths would be labeled with $0$, $2$, $4$, $6$, and $8$, and the other with $p$, $p \pm 2$, $p \pm 4$, $p \pm 6$, and $p \pm 8$, where $p$ is a prime or the negative of a prime, as shown in the second labeling in Figure \ref{2-odd-b}.  However, note that $p$ is adjacent to $0$, $4$, and $8$, which constitute all of the equivalence classes modulo $3$.  Thus one of these three distances must be divisible by $3$ and prime, and therefore exactly $3$.  So $p$ must be one of the numbers $-3$, $3$, $1$, $7$, $5$, or $11$.  But each of these values of $p$ yields a non-prime edge: if $p=-3$, then $L((p+4)(0))=1$ and $L((p-4)(2))=9$; if $p=3$ then $L((p)(4))=1$; if $p=1$ then $L((p)(0))=1$; if $p=7$ then $L((p)(8))=1$; if $p=5$ then $L((p)(4))=1$; if $p=11$ then $L((p-4)(8))=1$ and $L((p+4) (0))=15$.
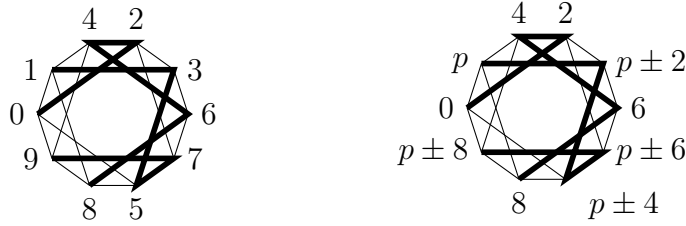
\begin{figure}
\begin{center}
\setlength{\unitlength}{0.5in}
\begin{tikzpicture}
 \draw (0.69,1.90) --(0.19,1.54)--(0,0.95)--(0.19,0.36)--(0.69,0)--(1.31,0)--(0,0.95);
 \draw (0.69,1.90)--(0.19,0.36);
 \draw (0.19,1.54)--(0.69,0);
 \draw (2,0.95)--(1.81,1.54)--(1.31,1.90)--(1.81,0.36)--(2,0.95);
 \draw[line width=0.075cm] (0,0.95)node[anchor=east]{0}--(1.31,1.90)node[anchor=south]{2}--(0.69,1.90)node[anchor=south]{4}--(2,0.95)node[anchor=west]{6}--(0.69,0)
 node[anchor=north]{8};
 \draw[line width=0.075cm] (0.19,1.54)node[anchor=east]{$1$}--(1.81,1.54)node[anchor=west]{$3$}--(1.31,0)node[anchor=north]{$5$}--(1.81,0.36)
 node[anchor=west]{$7$}--(0.19,0.36)node[anchor=east]{$9$};
\end{tikzpicture}\nolinebreak \hspace{2cm}
\begin{tikzpicture}(2,2)\thicklines
 \draw (0.69,1.90) --(0.19,1.54)--(0,0.95)--(0.19,0.36)--(0.69,0)--(1.31,0)--(0,0.95);
 \draw (0.69,1.90)--(0.19,0.36);
 \draw (0.19,1.54)--(0.69,0);
 \draw (2,0.95)--(1.81,1.54)--(1.31,1.90)--(1.81,0.36)--(2,0.95);
 \draw[line width=0.075cm] (0,0.95)node[anchor=east]{0}--(1.31,1.90)node[anchor=south]{2}--(0.69,1.90)node[anchor=south]{4}--(2,0.95)node[anchor=west]{6}--(0.69,0)
 node[anchor=north]{8};
 \draw[line width=0.075cm] (0.19,1.54)node[anchor=east]{$p$}--(1.81,1.54)node[anchor=west]{$p\pm 2$}--(1.31,0)node[anchor=north west]{$~p\pm 4$}--(1.81,0.36)
 node[anchor=west]{$p\pm 6$}--(0.19,0.36)node[anchor=east]{$p\pm 8$};
\end{tikzpicture}
\caption{A color-satisfying 2-odd labeling and a failed attempt at a color-satisfying prime-distance labeling of the same graph} \label{2-odd-b}\label{2-odd-a}
\end{center}
\end{figure}
\end{proof}

\begin{lemma} \label{red-cycle} Suppose $G$ is an edge-colored graph with a color-satisfying 2-odd labeling $L$.  Then every cycle in $G$ has a positive even number of blue edges.
\end{lemma}

\begin{proof}
Suppose $C=v_1 v_2 \ldots v_k$ is a cycle in $G$.  If $v_i v_{i+1}$ is red, then $L(v_i)$ and $L(v_{i+1})$ have the same parity, and if $v_i
v_{i+1}$ is blue, then $L(v_i)$ and $L(v_{i+1})$ have different parity.  The parity of the labels of a cycle must change an even number of
times, so $C$ has an even number of blue edges.  Now suppose by way of contradiction that $C$ has no blue edges.  Since $v_1 v_2$ is red, we may
assume without loss of generality that $L(v_1)=0$ and $L(v_2)=2$.  Then since vertex labels are not repeated, $L(v_3)=4$, $L(v_4)=6$, $\ldots,$
$L(v_k)=2k-2$.  Since $k \geq 3$, $L(v_1 v_k)=2k-2$ is neither 2 nor odd.
\end{proof}

\begin{corollary}In an edge-colored graph with a color-satisfying 2-odd labeling, every odd cycle has at least one red edge, and every 3-cycle has exactly one red edge. \label{odd-cycle}
\end{corollary}

Since every prime distance labeling is also a 2-odd labeling, Lemma~\ref{red-cycle} and Corollary~\ref{odd-cycle} hold for prime distance labelings as well.

The following lemma will be useful also.  Recall that the \textit{\textbf{symmetric difference}} $H_1 \triangle H_2$ of two subgraphs $H_1$ and $H_2$ of a graph $G$ is the graph with an edge $e$ in $H_1 \triangle H_2$ if $e$ is in $H_1$ or $H_2$, but not both \cite{West96}.  For convenience, we ignore isolated vertices in $H_1 \triangle H_2$.  Note that if $C_1$ and $C_2$ are cycles in $G$, then $C_1 \triangle C_2$ is a union of edge-disjoint cycles in $G$.
\bigskip

\begin{lemma} \label{sym-diff} If $C_1$ and $C_2$ are cycles in an edge-colored graph $G$ that both have a positive even number of blue edges, then
$C_1\triangle C_2$ has a positive even number of blue edges, as well. \end{lemma}

\begin{proof} If $C_1$ and $C_2$ have an odd number of blue edges in common, then they each have an odd number of blue edges that are not shared; thus,
the number of blue edges in $C_1 \triangle C_2$ is even.  If they have an even number of blue edges in common, then they also both have an
even number that are not shared, and $C_1 \triangle C_2$ again has an even number of blue edges. \end{proof}

Note that by the result of Eggleton, Erd{\H{o}}s, and Skilton above, $\chi(G) \leq 4$ if $G$ is a prime distance graph.  It follows that $K_n$ is not a prime distance graph if $n \geq 5$.  We prove that $K_n$ is also not a 2-odd graph for $n \geq 5$.
\bigskip

\begin{proposition}The graph $K_5$ is not 2-odd.\end{proposition} \label{K5}

 \begin{proof}We label the vertices of $K_5$ as $v_1$, $v_2$, $v_3$, $v_4$, and $v_5$ in cyclic order.  By way of contradiction, assume that $K_5$ has a 2-odd coloring.  By Lemma \ref{red-cycle}, the outside cycle of $K_5$ must have at least one red edge; without loss of generality, we may assume that $v_1 v_2$ is red.  Then by Corollary~\ref{odd-cycle}, $v_1 v_5$ and $v_2 v_5$ must be blue.  Similarly, $v_1 v_4$ and $v_2 v_4$ must be blue, as must $v_1 v_3$ and $v_2 v_3$.  With both $v_1 v_3$ and $v_1 v_4$ blue, we must have $v_3 v_4$ red again by Corollary~\ref{odd-cycle}.  But now, because of 3-cycles on $v_2$, $v_3$, $v_5$ and $v_3$, $v_4$, $v_5$, we must have $v_3 v_5$ and $v_4 v_5$ red, which creates a red 3-cycle on $v_3$, $v_4$, and $v_5$.
\end{proof}

Since every subgraph of a 2-odd graph is also 2-odd, it follows that $K_n$ is not 2-odd for $n \geq 5$.

\section{A Characterization of 2-odd Graphs}

The following theorem characterizes 2-odd graphs in terms of edge colorings.  In an edge-colored graph, the \textit{\textbf{red-degree}} (respectively, \textit{\textbf{blue-degree}}) of a vertex $v$ is the number of red edges (blue edges) incident with $v$.
\bigskip

\begin{theorem} \label{2-odd-characterization} A graph $G$ is 2-odd if and only if it admits an edge-coloring satisfying the following two conditions:
\begin{enumerate} \item No vertex of $G$ has red-degree greater than 2. \label{red-degree}
    \item Every cycle in $G$ contains a positive even number of blue edges. \label{blue-cycle}
\end{enumerate}
\end{theorem}

\begin{proof} If $G$ is 2-odd, then it admits a labeling with distinct integers such that the difference between labels on adjacent vertices is either odd or exactly 2.  This induces a coloring on the edges of $G$: when the difference between such labels is exactly 2, color the corresponding edge
red; when the difference is odd, color the edge blue.  No vertex can be incident with three or more red edges since such a vertex
would have a label exactly 2 different from three distinct integers.  Also, since the labeling is color-satisfying by construction, Lemma
\ref{red-cycle} implies that Condition~\ref{blue-cycle} is satisfied.

Conversely, suppose that $G$ admits an edge-coloring satisfying Conditions~\ref{red-degree} and~\ref{blue-cycle}.  Consider the subgraph $H$ of $G$ consisting solely of red edges.  By Conditions~\ref{red-degree} and~\ref{blue-cycle}, $H$ is a union of paths.  Consider the minor $G/H$ obtained by contracting all of the red edges.  We see that $G/H$ is bipartite since $G/H$ has no odd cycles by Condition~\ref{blue-cycle}.  Label the partite sets of $G/H$ $A$ and $B$.  Partition the vertices of $G$ into sets $A'$ and $B'$ depending on whether they contract to a vertex in $A$ or $B$ in $G/H$.  Note that since every red path in $G$ contracts to a vertex in $G/H$, each red path in $G$ is completely contained in either $A'$ or $B'$.  Label the vertices of each of these paths with consecutive odd integers if the path lies in $A'$ and with consecutive even integers if the path lies in $B'$, using only previously unused integers in each case.  Any remaining vertices of $A'$ again can be labeled with any unused odd integers and remaining vertices of $B'$ can be labeled with any unused even integers.  This labeling is a 2-odd labeling of $G$ since every edge between vertices of the same parity is labeled with $2$. \end{proof}

\section{Circulant Graphs}

For a positive integer $n \geq 3$ and subset $S \subseteq \{1,2, \ldots, n\}$, the \textit{\textbf{circulant graph}} $\Circ(n, S)$ is the graph with
vertex set $\{v_1,v_2, \ldots, v_n\}$ and an edge between vertices $v_i$ and $v_j$ if and only if $|i-j| \m n \in S$ \cite{Gross06}.  Equivalently, $\Circ(n,S)$ is the Cayley graph of the group $\mathbb{Z}_n$ with generating set $S$.  In this section we focus on the circulant graphs $\Circ(n, \{1,k\})$ for $1 \leq k \leq n-1$, which, for simplicity, we write as $\Circ(n,k)$.  Since $\Circ(n,k) \cong \Circ(n,n-k)$, we choose $k \leq n/2$.

\begin{figure}[ht!]
\begin{center}\includegraphics[width=2in]{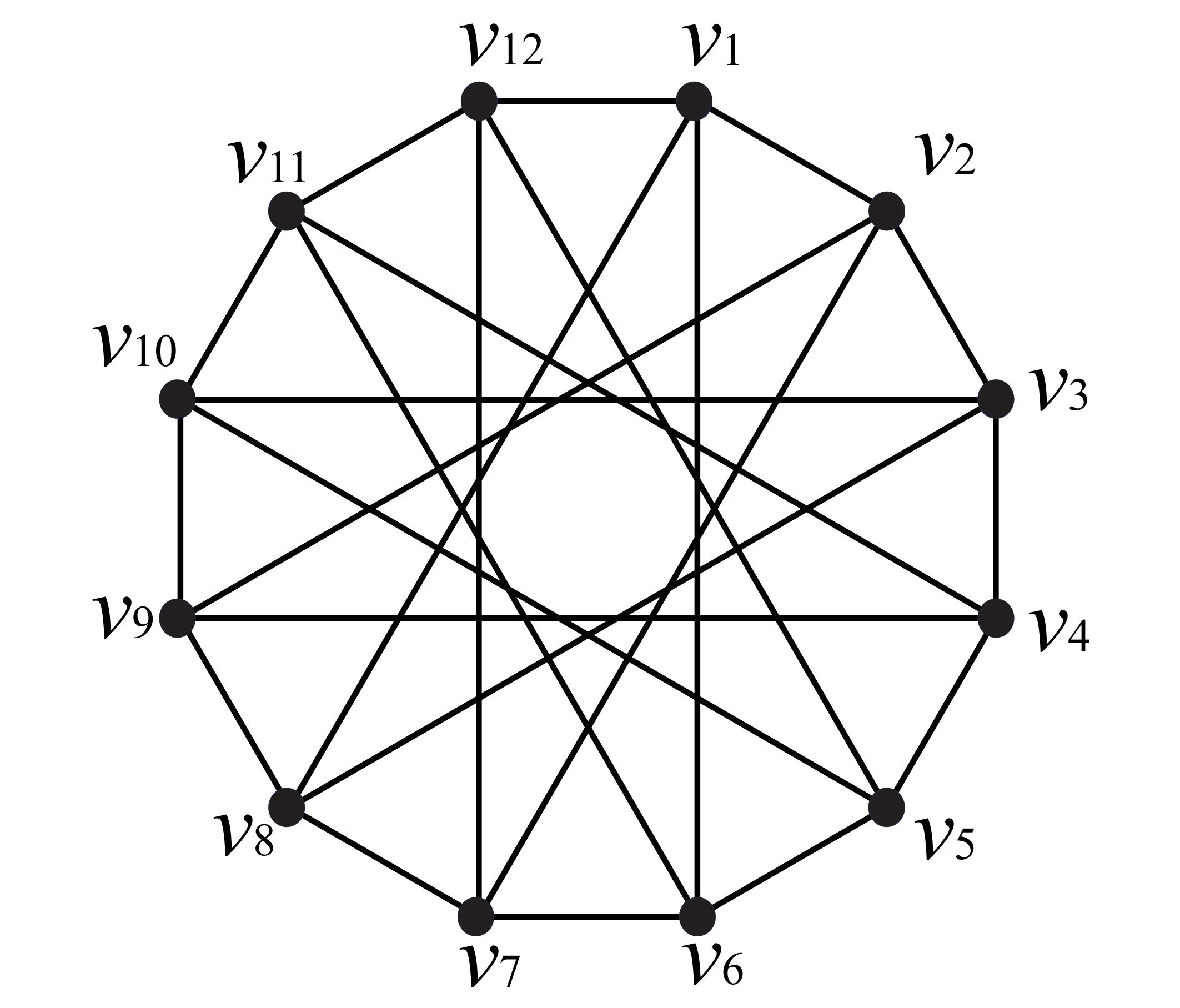}
\end{center}
\caption{The graph $\Circ(12,5)$.}\label{circ12-5}
\end{figure}

In these graphs we call the edges for which $|i-j| \m n =1$ the \textit{\textbf{outside edges}} and the edges for which $|i-j| \m n =k$ the \textit{\textbf{inside edges}}, since our drawings of $\Circ(n,k)$ will have the outside edges around the outside of a circle, with inside edges as chords.  Each vertex in $\Circ(n,k)$ is incident to exactly 2 outside edges and 2 inside edges as long as $k>1$.  If $k=1$ we say that $\Circ(n,k)$ has no inside edges.  In addition, we define a \textit{\textbf{generating cycle}} of $\Circ(n,k)$ to be either a $(k+1)$-cycle with exactly one inside edge, or the outside $n$-cycle.  We will refer to the generating $(k+1)$-cycles as $g_1, g_2, \ldots, g_n,$ where $g_i$ is the generating cycle with inside edge $v_i v_{i+k}$.

\subsection{2-odd Circulant Graphs}

\begin{lemma} \label{generate} Every cycle in $\Circ(n,k)$ is a generating cycle or a symmetric difference of generating cycles. \end{lemma}

\begin{proof} Let $C$ be a cycle in $\Circ(n,k)$.  The proof is by induction on the number of inside edges of $C$.  If there are none, then $C$ is the
outside cycle and the lemma is satisfied.  Now assume that $C$ has $c$ inside edges and that any cycle with fewer than $c$ inside edges satisfies the lemma.  If $C$ is a generating cycle, we are done.  Otherwise, let $D$ be a generating cycle whose inside edge also belongs to $C$.  Then $C\triangle D$ is an edge-disjoint union of cycles $C_1$, $C_2$, $\ldots$, $C_k$, each of which has fewer than $c$ inside edges and therefore satisfies the lemma.  Since $C_1 \triangle C_2 \triangle \cdots \triangle C_k \triangle D=(C\triangle D)\triangle D=C$, we see that $C$ also satisfies the lemma.

\end{proof}

\begin{theorem} \label{2-odd-circ-char} A circulant graph $\Circ(n,k)$ is 2-odd if and only if it admits an edge-coloring satisfying the following two conditions:
\begin{enumerate} \item No vertex of $\Circ(n,k)$ has red-degree greater than 2. \label{red-degree-circ}
    \item Every generating cycle of $\Circ(n,k)$ contains a positive even number of blue edges. \label{blue-cycle-circ}
\end{enumerate}
\end{theorem}

\begin{proof} Apply Lemmas \ref{sym-diff} and \ref{generate} and Theorem \ref{2-odd-characterization}.
\end{proof}

Theorem \ref{2-odd-circ-char} gives us a way to quickly verify whether a given edge-coloring of $\Circ(n,k)$ shows that $\Circ(n,k)$ is 2-odd.

The following theorem gives a characterization of 2-odd circulant graphs.
\bigskip

\begin{theorem} The circulant graph $\Circ(n,k)$ is 2-odd if and only if $(n,k)\ne (5,2).$
\end{theorem}

\begin{proof}

\noindent \textbf{Case 1. $k=1$.} In this case $\Circ(n,k)=C_n$, which by Theorem~\ref{cycle-theorem} is prime distance, hence 2-odd, for all $n$.
\bigskip

\noindent \textbf{Case 2. $n$ is even, $k$ is even, and $1< k < n/2$.} \label{even-even}  We construct a 2-odd coloring of $\Circ(n,k)$.  Color red the outside edges $v_n v_1$, $v_k v_{k+1}$, $v_{2k} v_{2k+1}$, $\ldots$, $v_{mk} v_{mk+1}$, where $m$ is the greatest odd integer such that $mk < n$.  That is, color red every $k$th outside edge, stopping after the maximum even number of outside edges are red without passing the first red edge.  Also color red the inside edge of any generating cycle that contains both $v_n v_1$ and $v_{mk} v_{mk+1}$ (the first and last outside edges that are red) or none of the red outside edges. Color the remaining edges blue.

We show that this coloring satisfies both conditions of Theorem~\ref{2-odd-circ-char}.  We first check Condition~\ref{blue-cycle-circ}.  Note that our construction guarantees an even number of red edges on the outside cycle, which is an even cycle since $n$ is even.  The remaining generating cycles are all odd, so to satisfy Condition~\ref{blue-cycle-circ} they must have an odd number of red edges.  Since none of these cycles contain more than 2 outside red edges, by construction they have an inside red edge if and only if they have an even number of outside red edges, so Condition~\ref{blue-cycle-circ} is satisfied.

Finally, we check Condition~\ref{red-degree-circ}. Since every vertex in $\Circ(n,k)$ is incident to 2 outside edges and 2 inside edges, and there are no 2 consecutive red outside edges in this coloring, any vertex with red-degree 3 or greater must be incident to exactly 1 outside red edge and 2 inside red edges.  However, the only red inside edges are those that belong to generating cycles containing both $v_n v_1$ and $v_{mk} v_{mk+1}$ or containing no outside red edges.  The latter case cannot create a red-degree-3 vertex.  For the former case, we observe that the only potential red-degree-three vertices are $v_n$, $v_1$, $v_{mk}$, and $v_{mk+1}.$  Of these, $v_n$ and $v_{mk+1}$ do not belong to generating cycles that contain both $v_n v_1$ and $v_{mk}v_{mk+1}.$  Since neither $v_1v_{k+1}$ nor $v_{(m-1)k} v_{mk}$ (the other inside edges on $v_1$ and $v_{mk}$) belongs to such a generating cycle either, none of the four candidates can have red-degree three.
\bigskip

\noindent \textbf{Case 3. $n$ is even, $k$ is even, and $k=n/2$.}  Coloring exactly the inside edges red gives each generating cycle (except the outside cycle) one red edge and $k=n/2$ blue edges, and the outside cycle $n$ blue edges.  Thus, this edge-coloring satisfies Theorem \ref{2-odd-circ-char}.
\bigskip

\noindent \textbf{Case 4. $n$ is even and $k$ is odd.}  Since every generating cycle has length $k+1$ or length $n$, every cycle of $G$ is even by Lemma~\ref{generate}, so coloring every edge of $G$ blue yields a 2-odd coloring of $G$.
\bigskip

\noindent \textbf{Case 5. $n$ is odd and $k$ is odd.} Color the outside edge $v_1 v_2$ red, and color the $k$ inside edges $v_{i-k} v_i$ red for $2 \leq i \leq k+1$.  This gives one red edge on the outside cycle and one red edge on each of the $k$ generating cycles that include it.  Note that the outside cycle is odd and the other generating cycles are even.  This coloring gives $1$ red edge on the outside cycle and either $0$ or $2$ red edges on each other generating cycle, satisfying the conditions of the theorem.
\bigskip

\noindent \textbf{Case 6. $n$ is odd, $k$ is even, and  $n>5$.} \label{odd-even} Apply the technique used in the proof of Case 2, but take $m$ as the greatest {\it even} integer such that $mk<n.$  The argument then proceeds as above.  Note that the argument fails for $K_5=\Circ(5,2)$ since the procedure would create a red 3-cycle.
\bigskip

\noindent \textbf{Case 7. $(n,k)=(5,2)$.}  In this case $\Circ(n,k)=K_5$, which by Proposition~\ref{K5} is not a 2-odd graph.

\end{proof}

We will use these 2-odd colorings to aid us in determining which circulant graphs are prime-distance graphs.

\subsection{Prime Distance Circulant Graphs}

The following theorem summarizes the results of this section.  Recall that without loss of generality, we choose $k \leq n/2$.
\bigskip

\begin{theorem}
  The circulant graph $\Circ(n,k)$ is not a prime distance graph if

  \begin{enumerate}
  \item $n$ is odd and $k=2$

  \item $n$ is odd and $k=(n-1)/2$.
  \end{enumerate}

  The circulant graph $\Circ(n,k)$ is a prime distance graph if

  \begin{enumerate}

    \item $n>7$ and $k=3$

    \item $n$ is even, $k=n/2$, and $k$ is even

    \item $n$ is even and $k$ is odd.

  \end{enumerate}
\end{theorem}

\noindent For the remaining values of $n$ and $k$ we make the following conjecture, which we have verified for $n \leq 14$.
\bigskip

\begin{conjecture} \label{circ-conjecture} The circulant graph $\Circ(n,k)$ is a prime distance graph if and only if none of the following hold:
\begin{enumerate}
\item $n$ is odd and $k=2$
\item $n$ is odd and $k=(n-1)/2$
\item $(n,k)=(6,2)$.
\end{enumerate}
\end{conjecture}
\bigskip

\begin{lemma} \label{CF}
The edge-colored graph $CF$ shown in Figure~\ref{fan} has exactly one color-satisfying prime distance labeling up to isomorphism, labeling the
vertices $a$, $b$, $c$, and $d$ with $0$, $2$, $4$, and $7$, respectively.
\end{lemma}

\begin{proof}
We construct a color-satisfying prime distance labeling $L$ of $CF$.  Since the edge $ab$ is red, we may assume $L(a)=0$ and $L(b)=2$. Since labels on
vertices are unique, this forces $L(c)=4$.

Suppose $L(d)=x$.  Then $L(A)=x$, $L(B)=|x-2|$, and $L(C)=|x-4|$ are all prime, so they must be the numbers 3, 5, and 7.  Since there exists an
automorphism swapping $A$ and $C$, we may choose $L(A)=7$, $L(B)=5$, and $L(C)=3$, so $L(d)=7$.  This is a valid color-satisfying prime distance
labeling of $CF$.
\end{proof}

\setlength{\cfk}{0.5cm}
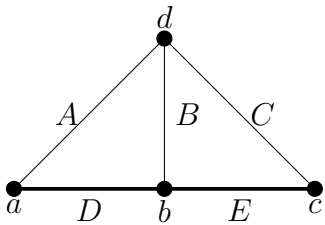
\begin{figure}[h!]
\begin{center}
\begin{tikzpicture}
\draw(0,0)node[anchor=north]{$a$}--(1,0)node[anchor=north]{$D$}--(2,0)node[anchor=north]{$b$}--(3,0)node[anchor=north]{$E$}--(4,0)node[anchor=north]{$c$}
--(3,1)node[anchor=west]{$C$}--(2,2)node[anchor=south]{$d$}--(1,1)node[anchor=east]{$A$}--(0,0);
\draw(2,0)--(2,1)node[anchor=west]{$B$}--(2,2);
\draw[line width=0.05 cm](0,0)--(2,0)--(4,0);
\draw[fill=black](0,0) circle(1mm);
\draw[fill=black](2,2) circle(1mm);
\draw[fill=black](2,0) circle(1mm);
\draw[fill=black](4,0) circle(1mm);
\end{tikzpicture}
\caption{The edge-colored graph $CF$} \label{fan}
\end{center}
\end{figure}

We now define the edge-colored graph $CF_k$ for all $k \geq 1$ ($CF$ stands for colored fan). $CF_k$ has vertices $\{a_1, a_2, \ldots, a_{k+3},$ $b_1, b_2, \ldots, b_k\}$ and is comprised of three paths and two extra edges: the red-edge path $a_1, a_2, \ldots, a_{k+3}$, the red-edge path $b_1, b_2, \ldots, b_k$, the blue-edge path $a_2, b_1, a_3, b_2, \ldots, a_{k+1}, b_k,$ $a_{k+2}$, and the two blue edges $a_1 b_1$ and $a_{k+3} b_k$.  The graphs $CF_1$, $CF_2$, and $CF_k$ are shown in Figure~\ref{CFk}.

\begin{figure}[h!]
\begin{center}\includegraphics[width=4in]{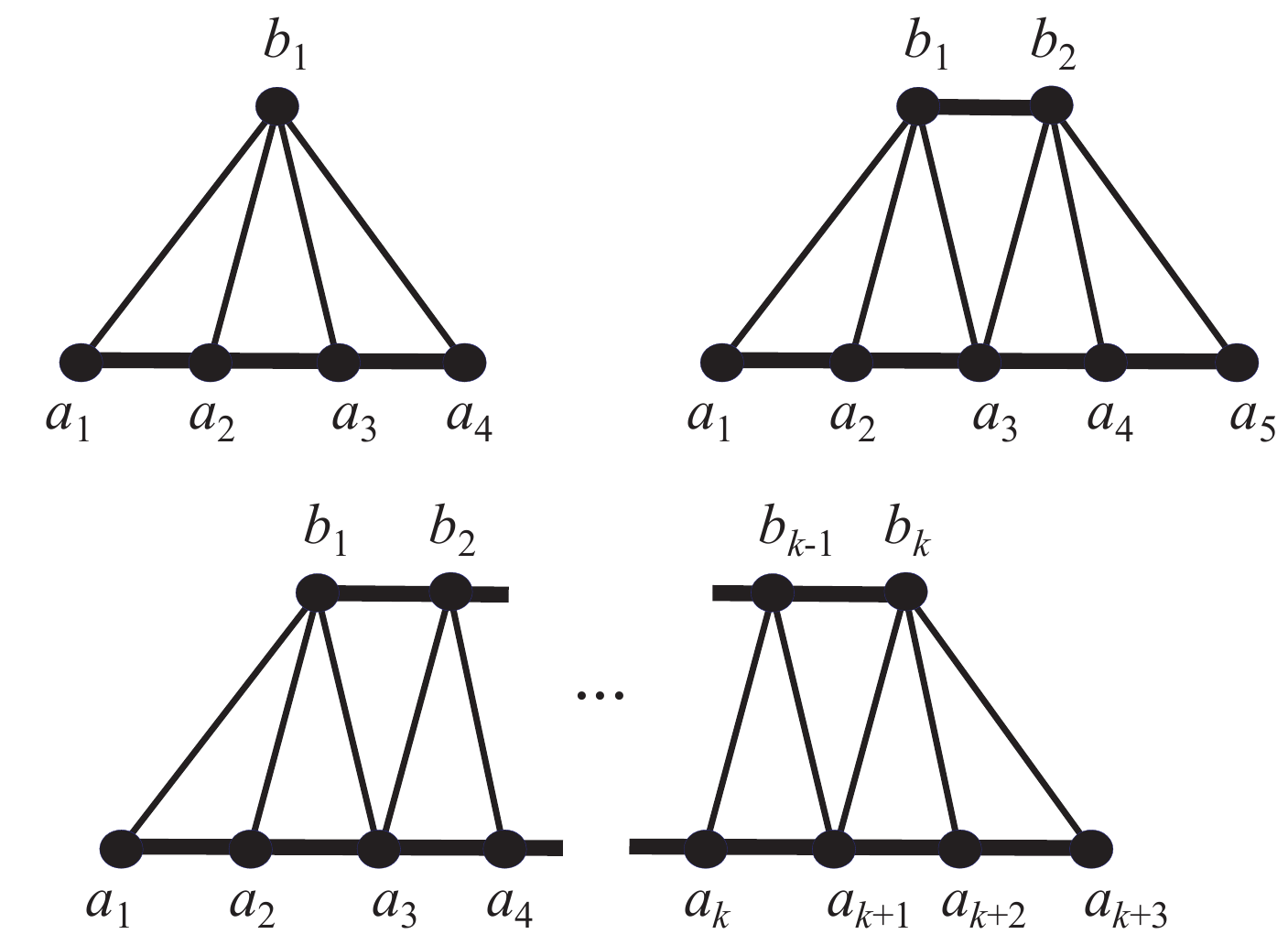}
\end{center}
\caption{The edge-colored graphs $CF_1$, $CF_2$, and $CF_k$.} \label{CFk}
\end{figure}
\bigskip

\begin{lemma}
$CF_k$ has no color-satisfying prime distance labeling for all $k \geq 1$.
\end{lemma}

\begin{proof}
By way of contradiction, suppose that $CF_k$ has a color-satisfying prime distance labeling $L$.  The induced subgraph of $CF_k$ on the vertices $a_1$, $a_2$, $a_3$, and $b_1$ is $CF$, so by Lemma~\ref{CF}, $L(a_1)=0$, $L(a_2)=2$, $L(a_3)=4$, and $L(b_1)=7$.  In the case $k=1,$ this forces $a_4$ to have label 6, which makes $a_4b_1$ non-prime.  Now assume $k>1.$

Since edge $b_1 b_2$ is red, $L(b_2)$ is either 5 or 9.  But if $L(b_2)=5$ then $L(a_3 b_2)=1$, which is not prime, so $L(b_2)=9$.  Since the edges $a_3 a_4$, $a_4 a_5$, $\ldots,$ $a_{k+2} a_{k+3}$ and the edges $b_2 b_3$, $b_3 b_4$, $\ldots$, $b_{k-1} b_k$ are all red, this forces the labeling on the remaining vertices $L(a_4)=6$, $L(a_5)=8$, $\ldots,$ $L(a_{k+3})=2k+4$ and $L(b_3)=11$, $L(b_4)=13$, $\ldots$, $L(b_k)=2k+5$.  But this means $L(a_{k+3}b_k)=1$, which is not prime.

\end{proof}

\begin{theorem}
$\Circ(2n+1,2)$ is not a prime distance graph for all $n \geq 2$.
\end{theorem}

\begin{proof}
Suppose by way of contradiction that $L$ is a prime distance labeling of $\Circ(2n+1, 2)$.  Since the outside edges of $\Circ(2n+1, 2)$ form an
odd cycle and the inside edges form an odd cycle, by Lemma~\ref{odd-cycle}, $\Circ(2n+1,2)$ must have at least one red outside edge and at least one red inside edge.

Suppose without loss of generality that the edge $v_1 v_2$ is red, and $L(v_1)=0$ and $L(v_2)=2$. Note that $v_2 v_3$ must be blue by
Lemma~\ref{red-cycle} since vertices $v_1$, $v_2$, and $v_3$ form a 3-cycle. Suppose $v_i v_{i+1}$ is the next red outside edge, i.e. $v_i
v_{i+1}$ is red and $v_j v_{j+1}$ is blue for all $2 \leq j \leq i-1$. Consider the induced subgraph $G$ of $\Circ(2n+1, 2)$ on the vertices
$v_1, \ldots, v_{i+1}$. If $i$ is even (and thus $i\ge 4$), we will show that the coloring on the edges of $G$ induced by the prime distance labeling $L$ of $\Circ(2n+1, 2)$ yields the edge-colored graph $CF_{i/2-1}$.  Since $CF_{i/2-1}$ has no color-satisfying prime distance labeling, this will prove that $i$ must be odd.

Since $v_1 v_2$ is red, $v_1 v_3$ must be blue by Corollary~\ref{odd-cycle}.  Since $v_i v_{i+1}$ is red, $v_{i-1} v_{i+1}$ must be blue by
Corollary~\ref{odd-cycle}.  For all $2 \leq j \leq i-2$, since $v_j v_{j+1}$ and $v_{j+1} v_{j+2}$ are blue, $v_j v_{j+2}$ must be red by
Corollary~\ref{odd-cycle}.  This is exactly the edge-coloring of the graph $CF_{i/2-1}$, with $v_1=a_1$, $v_2=a_2$, $v_3=b_1$, $v_4=a_3$, $v_5=b_2$,
$v_6=a_4$, $\ldots,$ $v_{i-1}=b_k$, $v_i=a_{k+2}$, $v_{i+1}=a_{k+3}$.

Thus $i$ must be odd.  Note that vertices with a red edge between them have labels with the same parity, and vertices with a blue edge between
them have labels with different parity.  Since the number of blue edges between $v_2$ and $v_i$ is $i-2$, which is odd, $L(v_i)$ is odd, so
$L(v_{i+1})$ is also odd. Continuing in this way, if $v_j v_{j+1}$ is the next red edge, $L(v_j)$ is even and $L(v_{j+1})$ is even.  Thus the
labels on the vertices of the outside red edges alternate parity in pairs. This implies that there are an even number of outside red edges,
contradicting Lemma~\ref{odd-cycle}.

\end{proof}

Since we know that $\Circ(2n+1,2)$ is a 2-odd graph if $n>2$, we also obtain the following corollary.
\bigskip

\begin{corollary} Not every 2-odd graph is a prime distance graph. \label{odd-not-prime} \end{corollary}

%\subsubsection{$\Circ(2n+1,n)$}

A {\bf\textit{ bowtie path}} in $\Circ(2n+1,n)$ is a path of the form $$v_a v_{a+n+1} v_{a+n+2} v_{a+n+3} \ldots v_{a+n+r} v_{a+r}$$ for some positive integers $a$ and $r$.  Figure \ref{bowtie-path} shows an example of a bowtie path.

\begin{figure}
\begin{center}
\includegraphics[width=2.5in]{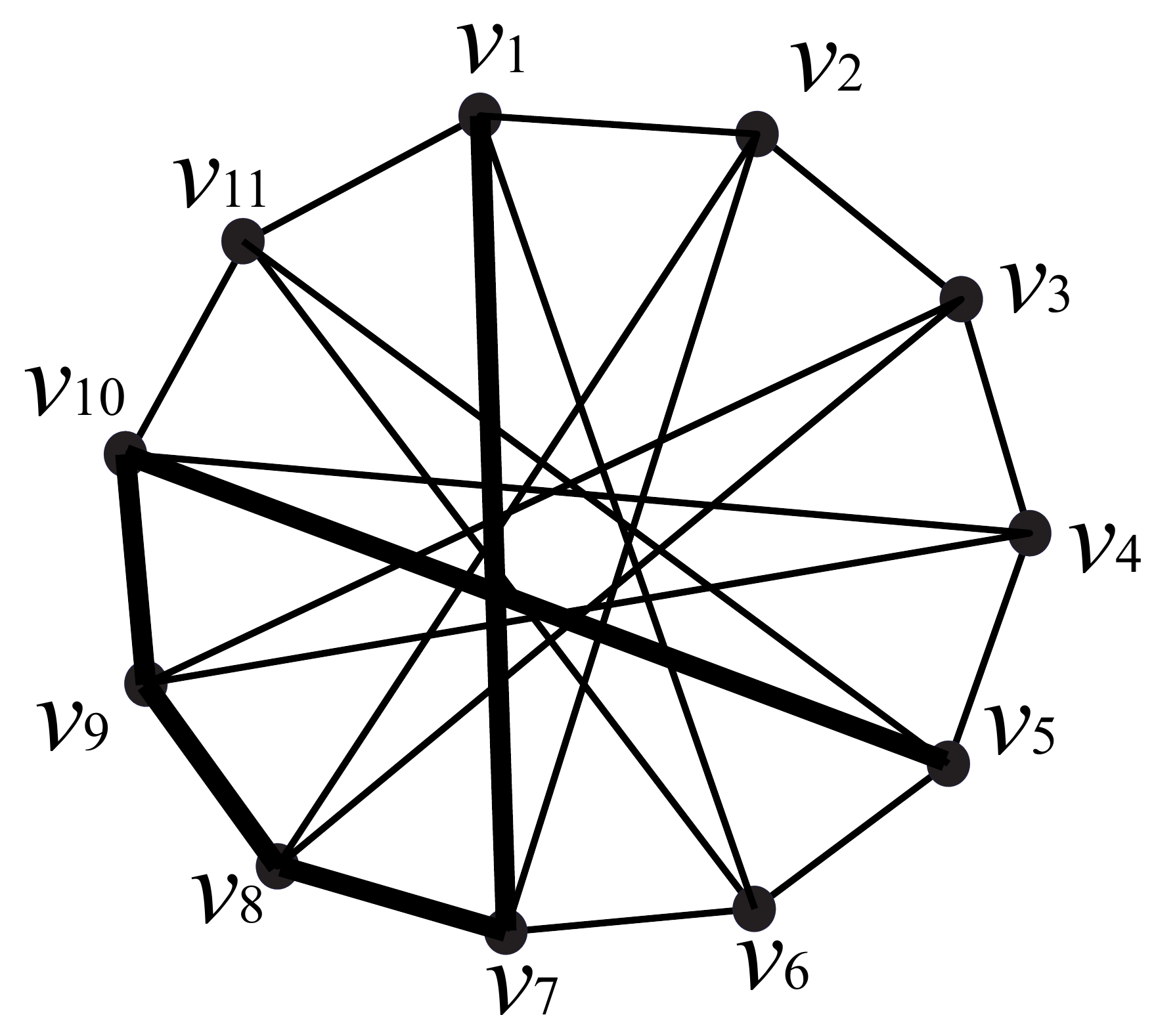}
 \caption{A bowtie path in $\Circ(2n+1,n)$, with $n=5$, $a=1$, and $r=4$.}\label{bowtie-path}
\end{center}
\end{figure}
\bigskip

\begin{lemma} \label{opposite-reds} If in a 2-odd coloring of $\Circ(2n+1,n)$, the edges $v_a v_{a+1}$ and $v_{a+1} v_{a+2}$ are red, then so is the edge $v_{a+n+1} v_{a+n+2}$. \end{lemma}

\begin{proof} Without loss of generality, let $a=1$.  Since $v_1 v_2$ is red, we must have both $v_1 v_{n+2}$ and $v_2 v_{n+2}$ blue since $v_1$, $v_2$, and $v_{n+2}$ form a triangle.  Similarly, since $v_2 v_3$ is red, $v_2 v_{n+3}$ and $v_3 v_{n+3}$ are blue.  But then we have $v_2 v_{n+2}$ and $v_2 v_{n+3}$ both blue, so $v_{n+2} v_{n+3}$ must be red since it is the third edge in the triangle with vertices $v_2$, $v_3$, and $v_{n+3}$. \end{proof}

\begin{lemma} \label{no-bowtie} If $\Circ(2n+1,n)$ is edge-colored so that a bowtie path in $\Circ(2n+1,n)$ has only red edges, then it has no color-satisfying prime-distance labeling. \end{lemma}

\begin{proof} Suppose by way of contradiction that $\Circ(2n+1,n)$ has such a prime-distance labeling.  Without loss of generality, let $a=1$, so the all-red bowtie path in $\Circ(2n+1,n)$ is on the vertices $v_1$,$v_{n+2}$, $v_{n+3}$, $\ldots$, $v_{n+r+1}$, $v_{r+1}$.  Again without loss of generality, we may assume that these vertices are labeled $0$, $2$, $4$, $6$, $\ldots$, and $2(r+1)$.  By Lemma \ref{opposite-reds}, the edges
$v_{n+3+n} v_{n+4+n}=v_2 v_3$, $\ldots$, $v_{n+r+n} v_{n+r+1+n}=v_{r-1} v_r$ are also red.  Note that if $r=2$ then there are no edges of this form.  In either case, the edge $v_1 v_2$ must be blue since it belongs to the triangle with vertices $v_1$, $v_2$, $v_{n+2}$ which already has the red edge $v_1 v_{n+2}$.

Consider the possible labels on the vertex $v_2$.  This vertex is adjacent to vertices with labels $0$, $2$, and $4$, which represent all
equivalence classes modulo three.  Thus one of the distances between $v_2$ and these three neighbors must be a multiple of 3 and prime, so it must be exactly three.  Thus, $v_2$ must be labeled with either $-3$ or $7$.  Now we may proceed along the new red path given by Lemma \ref{opposite-reds}.

The next label along this path after a $-3$ could be $-5$ since it is adjacent to a vertex labeled $4$, and the next label after a $7$ could be $5$ for the same reason.  Thus the red path beginning at $v_2$ must have labels $-3$, $-1$, $1$, $\ldots$, $2r-7$ or $7$, $9$, $11$, $\ldots$, $2r+3$.  But the terminal vertex in this path, which is labeled either $2r-7$ or $2r+3$, is adjacent to $v_{r+1}$, which is labeled with $2(r+1)$.  The distance between these two vertices is thus either $9$ or $1$, neither of which is prime. \end{proof}

Define the \textit{\textbf{red path}} $RP_m$ and the \textit{\textbf{blue path}} $BP_m$ as the edge-colored subgraphs of $\Circ(2n+1,n)$ shown in Figure~\ref{red-path}.
\bigskip

\begin{theorem} $\Circ(2n+1,n)$ is not a prime distance graph for all $n \geq 2$. \end{theorem}

\begin{proof} Suppose by way of contradiction that $\Circ(2n+1,n)$ is a prime distance graph, and consider a prime distance labeling of $\Circ(2n+1,n)$, with corresponding edge coloring.  Since $2n+1$ is odd, the outside cycle must have at least one red edge.
\bigskip

\noindent \textbf{Step 1.} Every maximal red outside path in $\Circ(2n+1,n)$ with vertices $v_i$, $\ldots$, $v_{i+m}$ induces the edge-colored subgraph $RP_m$ shown in Figure~\ref{red-path} as follows:

Since this red path is maximal, the edges $v_{i-1} v_i$ and $v_{m+i} v_{m+i+1}$ are blue.  By Corollary~\ref{odd-cycle}, this induces blue edges $v_i v_{n+i+1}$, $v_{i+1} v_{n+i+1}$, $v_{i+1} v_{n+i+2}$, $v_{i+2} v_{n+i+2}$, $\ldots$, $v_{m+i-1} v_{m+i+n}$, $v_{m+i} v_{m+i+n}$
%corresponding to triangles with bases $v_1 v_2$, $v_2 v_3$, $\ldots$, $v_{m-1} v_m$, respectively.  Also, we have
and red edges $v_{n+i+1} v_{n+i+2},$ $v_{n+i+2} v_{n+i+3}$, $\ldots$, $v_{m+i+n-1} v_{m+i+n}$, as well as $v_{i-1}v_{n+i}.$
%This gives a red path of length $m+n-(n+2)=m-2.$  See Figure \ref{coloring-01}.
By Lemma \ref{no-bowtie}, we cannot have both edges $v_{m+i} v_{m+i+n+1}$ and $v_i v_{n+i}$ red, so we may assume without loss of generality, that $v_i v_{n+i}$ is blue.  Since $v_i v_{n+i+1}$ is also blue, we must have $v_{n+i} v_{n+i+1}$ red.  If $v_{m+i} v_{m+i+n+1}$ were blue as well, then, because $v_{m+i} v_{m+i+n}$ is blue, we would additionally have $v_{m+i+n} v_{m+i+n+1}$ and $v_{m+i+n+1} v_{m+i+1}$ red, contradicting Lemma~\ref{no-bowtie}.
%giving us a red path with $m+1$ vertices along the outside cycle and contradicting the maximality of our original red path.
Thus $v_{m+i} v_{m+i+n+1}$ must be red and $v_{m+i+n} v_{m+i+n+1}$ must be blue. %and $v_n v_{n+1}$ are blue.

\begin{figure}[h!]
\begin{center}
\includegraphics[width=\textwidth]{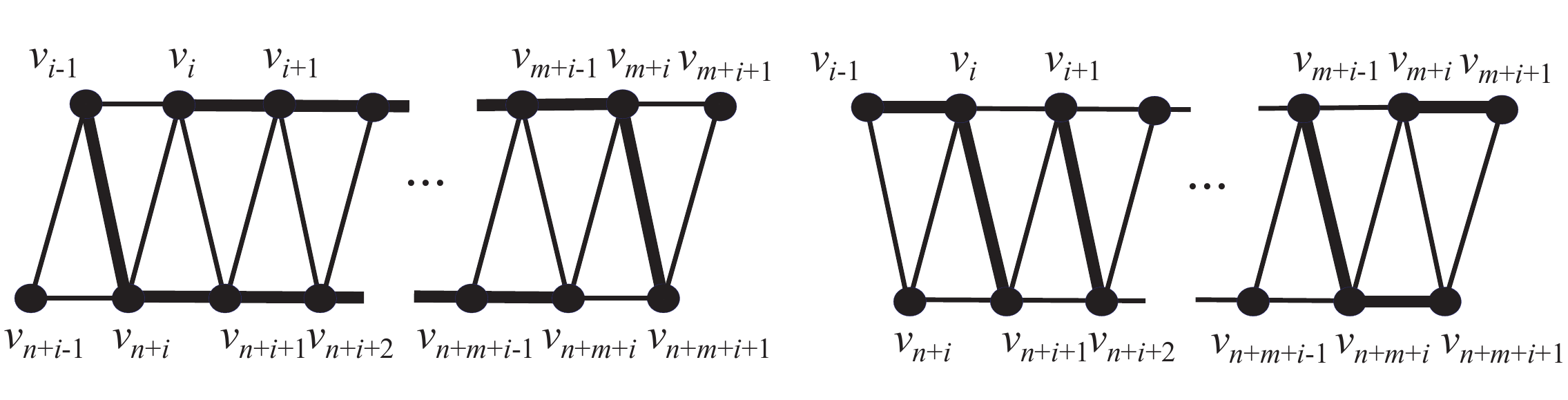}
\caption{The edge-colored graphs $RP_m$ (left) and $BP_m$ (right).} \label{red-path} \label{blue-path}
\end{center}
\end{figure}

\bigskip

\noindent \textbf{Step 2.} Every maximal blue outside path in $\Circ(2n+1,n)$ with vertices $v_i$, $\ldots$, $v_{m+i}$ induces the edge-colored subgraph $BP_m$ shown in Figure~\ref{blue-path}.

Since this blue path is maximal, the edges $v_{i-1} v_i$ and $v_{m+i} v_{m+i+1}$ are red.  By Step 1, either the edge $v_i v_{n+i+1}$ or the edges $v_{n+i} v_{n+i+1}$ and $v_{n+i+1} v_{i+1}$ are red.  Assume $v_i v_{n+i+1}$ is red (the other case is a reflection of this one).  Then by Corollary~\ref{odd-cycle}, $v_{i+1} v_{n+i+1}$ is blue.  If $v_{i+1} v_{i+2}$ is red, then $v_{i+1} v_{n+i+2}$ is blue, $v_{n+i+1} v_{n+i+2}$ is red, and we have $BP_1$.  Otherwise we continue as before, until reaching the first red edge $v_{m+i} v_{m+i+1}$.
\bigskip

\noindent \textbf{Step 3.} The subgraphs $RP_m$ and $BP_m$ cannot fit together to form a valid red-blue edge coloring of $\Circ(2n+1,n)$: to complete the outside cycle (after gluing together whatever copies of $BP_m$ and $RP_m$ appear), the top left corner of $RP_m$ must be adjacent to the bottom right corner of $BP_m.$  However, this results in a red bowtie path: $$\underbrace{v_{m+i-1}v_{n+m+i}v_{n+m+i+1}\ldots v_{n+m+i+r}}_{\mbox{from end of $BP_m$}}\underbrace{v_{n+m+i+r+1}v_{m+i+r}}_{\mbox{from beginning of $RP_m$}}.$$

Therefore $\Circ(2n+1,n)$ is not a prime-distance graph.
\end{proof}

%\subsection{Prime-Distance Circulant Graphs}

\begin{proposition}If $n$ is even and $k$ is odd, then $\Circ(n,k)$ is a prime distance graph. \end{proposition}

\begin{proof} Since $\Circ(n,k)$ is bipartite if $n$ is even and $k$ is odd, the result follows from Theorem~\ref{bipartite}. \end{proof}

\begin{theorem} If $n>7$, then $\Circ(n,3)$ is a prime distance graph. \label{n3-theorem}
\end{theorem}

\begin{proof}
By the previous Proposition, it suffices to consider $\Circ(2n+1,3).$  The labeling depends on the congruence class of $n$ modulo 3.
\bigskip

\noindent \textbf{Case 1. $n\equiv 0 \pmod 3.$}  In this case we label the vertices $v_1$, $v_2$, $v_3$, $v_4$, $v_5$, $v_6$, $\ldots$, $v_{3i+1}$, $v_{3i+2}$, $v_{3i+3}$, $\ldots$, $v_{n-3}$, $v_{n-2}$, $v_{n-1}$, $v_n$, $v_{n+1}$, $v_{n+2}$, $v_{n+3}$, $v_{n+4},$ $v_{n+5}$, $\ldots$, $v_{3j}$, $v_{3j+1}$, $v_{3j+2}$, $\ldots$, $v_{2n-3}$, $v_{2n-2}$, $v_{2n-1}$, $v_{2n}$, and $v_{2n+1}$ with the labels 4, 2, $-1$, 6, 9, 12, $\ldots$, $11i-5$, $11i-2$, $11i+1$, $\ldots$, $\f{11}{3}n-21$, $\f{11}{3}n-16$, $\f{11}{3}n-13$, $\f{11}{3}n-2$, $\f{11}{3}n+1$, $\f{11}{3}n+4$, $\f{11}{3}n-9$, $\f{11}{3}n-12$, $\f{11}{3}n-15$, $\ldots$, $\f{22}{3}n-11j+2$, $\f{22}{3}n-11j-1$, $\f{22}{3}n-11j-4$, $\ldots$, 13, 10, 7, 0, and $-3,$ respectively.  Note that the edges around the outside cycle appear in the pattern $3, 3, 5, 3, 3, 5, \ldots$, except for the edges between vertices $v_{2n-1}$ through $v_4$ and vertices $v_{n-1}$ through $v_{n+3}$.  Hence most inner edges are labeled 11.  The exceptions are easily verified to be prime also.
\bigskip

\noindent \textbf{Case 2. $n\equiv 1 \pmod 3.$}  In this case we keep the majority of the labels from Case 1, but change the labels of the vertices $v_{n-3}$, $v_{n-2}$, $v_{n-1}$, $v_n$, $v_{n+1}$, $v_{n+2}$, $v_{n+3}$, $v_{n+4},$ $v_{n+5}$, $v_{n+6}$, $\ldots$, $v_{3j}$, $v_{3j+1}$, and $v_{3j+2}$ to $\f{11(n-1)}{3}-16$, $\f{11(n-1)}{3}-13$, $\f{11(n-1)}{3}-10$, $\f{11(n-1)}{3}+7$, $\f{11(n-1)}{3}+10$, $\f{11(n-1)}{3}-7$, $\f{11(n-1)}{3}-4$, $\f{11(n-1)}{3}-9$, $\f{11(n-1)}{3}-12$, $\f{11(n-1)}{3}-15$, $\ldots$, $\f{22(n-1)}{3}-11j+10$, $\f{22(n-1)}{3}-11j+7$, and $\f{22(n-1)}{3}-11j+2$, respectively.  Again edges around the outside cycle appear in the pattern $3, 3, 5, 3, 3, 5, \ldots$, and the exceptions are easily verified to be prime also.
\bigskip

\noindent \textbf{Case 3. $n\equiv 2 \pmod 3.$}  In this case again we keep the majority of the labels from Case 1, but change the labels of the vertices $v_{n-3}$, $v_{n-2}$, $v_{n-1}$, $v_n$, $v_{n+1}$, $v_{n+2}$, $v_{n+3}$, $v_{n+4}$, $v_{n+5}$, $\ldots$, $v_{3j}$, $v_{3j+1}$, and $v_{3j+2}$ to $\f{11(n-2)}{3}-11$, $\f{11(n-2)}{3}-8$, $\f{11(n-2)}{3}-5$, $\f{11(n-2)}{3}+6$, $\f{11(n-2)}{3}+9$, $\f{11(n-2)}{3}+2$, $\f{11(n-2)}{3}-1$, $\f{11(n-2)}{3}-4$, $\f{11(n-2)}{3}-9$, $\ldots$, $\f{22(n-2)}{3}-11j+18$, $\f{22(n-2)}{3}-11j+13$, and $\f{22(n-2)}{3}-11j+10$, respectively.  Again the proof goes through as in Case 1.
\end{proof}

The proof of Theorem~\ref{n3-theorem} illustrates the difficulty in characterizing prime distance circulant graphs.  We believe this technique could be modified to prove $\Circ(n,k)$ is a prime distance graph for a fixed $k$ strictly between $2$ and $(n-1)/2$, but a new technique is needed to prove this for general $k$.
\bigskip

\begin{theorem} If $n>2$ and $n/2$ are even, then $\Circ(n,n/2)$ is a prime distance graph.
\end{theorem}

\begin{proof} $\Circ(4,2)$ and $\Circ(8,4)$ are easily shown to be prime distance, so assume $n\ge 12.$  Choose a prime $p$ sufficiently large such that by Vinogradov's Theorem, $p-3n/2+10$ can be expressed as a sum of three primes $q_1$, $q_2$, $q_3$.  That is, $p=3n/2+10+q_1+q_2+q_3$.

In counterclockwise order, we label $v_n$ with $0$, and $v_{n-1}$, $v_{n-2}$, $v_{n-3}$, $\ldots$, $v_{n/2+4}$ with $7$, $10$, $13$, $\ldots$, $3n/2-8$, respectively, label $v_{n/2}$ with $2$, and $v_{n/2-1}$, $v_{n/2-2}$, $v_{n/2-3}$, $\ldots$, $v_4$ with $5$, $8$, $11$, $\ldots$, $3n/2-10$, respectively.

Then we label $v_1$ with $p$, $v_2$ with $q_1+q_2+3k-10$, $v_3$ with $q_1+3k-10$, $v_{n/2+1}$ with $p+2$, $v_{n/2+2}$ with $q_1+q_2+3n/2-8$, and $v_{n/2+3}$ with $q_1+3n/2-8$.  Figure~\ref{middle-case} shows the case $n=10$ and $p=47$.

This gives a prime distance labeling of $G$.
\begin{figure}
\begin{center}\includegraphics[width=2.5in]{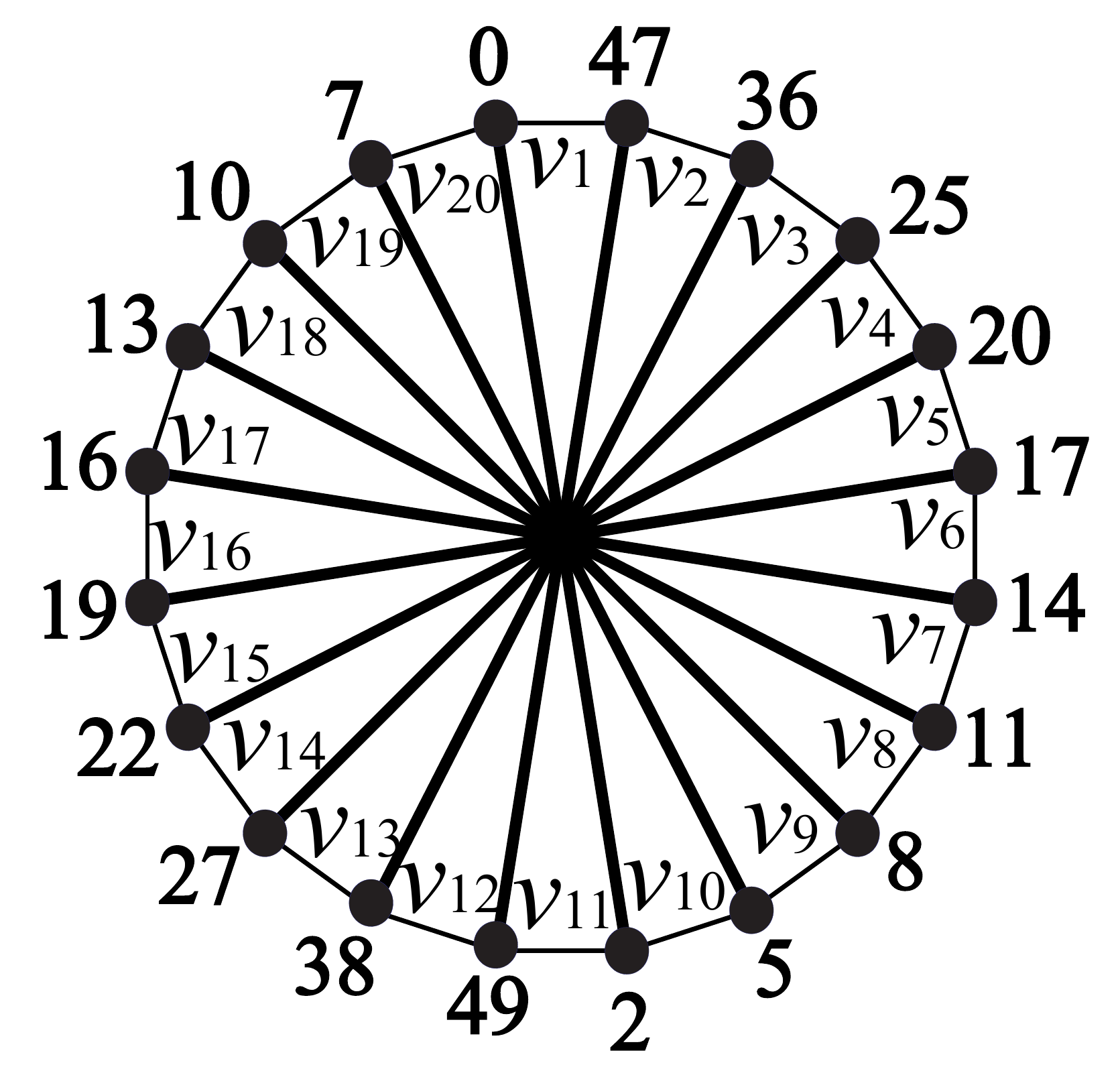}\end{center}
\caption{A prime distance labeling of $\Circ(20,10)$}\label{middle-case}
\end{figure}
\end{proof}

\section{Open Questions}

In addition to Conjecture~\ref{circ-conjecture}, we pose the following open questions.

\begin{enumerate}
\item Is there a family of graphs which are prime distance graphs if and only if Goldbach's Conjecture is true?

\item What circulant graphs $\Circ(n, S)$, for more general sets $S$, are prime distance graphs?

\item By our definition of prime distance graphs, $L(uv)$ may still be prime if $uv$ is not an edge of $G$.  How do the results of this paper change if we define $L(uv)$ to be prime if and only if $uv$ is an edge of $G$?

\item What other families of graphs are prime distance graphs?   More specifically, by Eggleton, Erd{\H{o}}s, and Skilton's result, all prime distance graphs have chromatic number at most 4, but not all planar graphs are prime distance graphs by the example shown in Figure~\ref{planar-not-PD}.  Can we classify which planar graphs are prime distance graphs?

\begin{figure}
\begin{center}\includegraphics[width=1.5in]{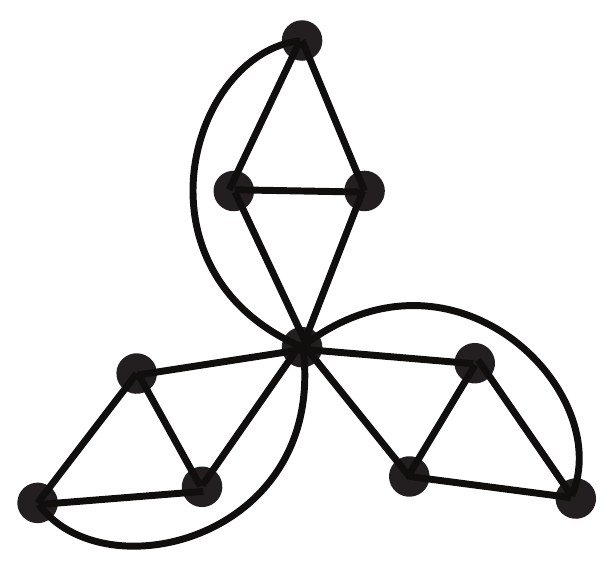}\end{center}
\caption{A planar graph which is not a prime distance graph.}\label{planar-not-PD}
\end{figure}

\end{enumerate}

\section{Acknowledgements}

The authors thank Xuding Zhu for suggesting the example in Figure~\ref{planar-not-PD}, and an anonymous reviewer for a careful and thorough reading of the paper.

\def\cprime{$'$}


\begin{thebibliography}{10}

\bibitem{Burton11}
David~M. Burton.
\newblock {\em Elementary Number Theory}.
\newblock McGraw-Hill, New York, NY, 7th edition, 2011.

\bibitem{polignac}
Alphonse de~Polignac.
\newblock Six propositions arithmologiques d\'{e}duites de crible
  d\'{E}ratosth\`{e}ne.
\newblock {\em Nouv. Ann. Math.}, 8:423--429, 1849.

\bibitem{Dickson66}
Leonard~Eugene Dickson.
\newblock {\em History of the theory of numbers. {V}ol. {I}: {D}ivisibility and
  primality.}
\newblock Chelsea Publishing Co., New York, 1966.

\bibitem{Eggleton85}
R.~B. Eggleton, P.~Erd{\H{o}}s, and D.~K. Skilton.
\newblock Colouring the real line.
\newblock {\em J. Combin. Theory Ser. B}, 39(1):86--100, 1985.

\bibitem{Eggleton86}
R.~B. Eggleton, P.~Erd{\H{o}}s, and D.~K. Skilton.
\newblock Erratum: ``{C}olouring the real line'' [{J}. {C}ombin.\ {T}heory
  {S}er.\ {B} {\bf 39} (1985), no.\ 1, 86--100; {MR}0805458 (87b:05057)].
\newblock {\em J. Combin. Theory Ser. B}, 41(1):139, 1986.

\bibitem{Eggleton90}
R.~B. Eggleton, P.~Erd{\H{o}}s, and D.~K. Skilton.
\newblock Colouring prime distance graphs.
\newblock {\em Graphs Combin.}, 6(1):17--32, 1990.

\bibitem{Eggleton88}
Roger~B. Eggleton.
\newblock New results on {$3$}-chromatic prime distance graphs.
\newblock {\em Ars Combin.}, 26(B):153--180, 1988.

\bibitem{Erdos66}
P.~Erd{\H{o}}s, A.~R{\'e}nyi, and V.~T. S{\'o}s.
\newblock On a problem of graph theory.
\newblock {\em Studia Sci. Math. Hungar.}, 1:215--235, 1966.

\bibitem{Gallian98}
Joseph~A. Gallian.
\newblock A dynamic survey of graph labeling.
\newblock {\em Electron. J. Combin.}, 5:Dynamic Survey 6, 43 pp.\ (electronic),
  1998.

\bibitem{green08}
Ben Green and Terence Tao.
\newblock The primes contain arbitrarily long arithmetic progressions.
\newblock {\em Ann. of Math. (2)}, 167(2):481--547, 2008.

\bibitem{Gross06}
Jonathan~L. Gross and Jay Yellen.
\newblock {\em Graph theory and its applications}.
\newblock Discrete Mathematics and its Applications (Boca Raton). Chapman \&
  Hall/CRC, Boca Raton, FL, second edition, 2006.

\bibitem{Helfgott13}
H.~A. Helfgott.
\newblock Major arcs for {G}oldbach's theorem.
\newblock arXiv:1305.2897.

\bibitem{Ramachandra97}
K.~Ramachandra and A.~Sankaranarayanan.
\newblock Vinogradov's three primes theorem.
\newblock {\em Math. Student}, 66(1-4):27--72, 1997.

\bibitem{Ramare95}
Olivier Ramar{\'e}.
\newblock On \v {S}nirel\cprime man's constant.
\newblock {\em Ann. Scuola Norm. Sup. Pisa Cl. Sci. (4)}, 22(4):645--706, 1995.

\bibitem{Rosen10}
Kenneth~H. Rosen.
\newblock {\em Elementary Number Theory and its Applications}.
\newblock Addison Wesley, Reading, MA, 6th edition, 2010.

\bibitem{Vinogradov04}
I.~M. Vinogradov.
\newblock {\em The method of trigonometrical sums in the theory of numbers}.
\newblock Dover Publications Inc., Mineola, NY, 2004.
\newblock Translated from the Russian, revised and annotated by K. F. Roth and
  Anne Davenport, Reprint of the 1954 translation.

\bibitem{Voigt94}
M.~Voigt and H.~Walther.
\newblock Chromatic number of prime distance graphs.
\newblock {\em Discrete Appl. Math.}, 51(1-2):197--209, 1994.
\newblock 2nd Twente Workshop on Graphs and Combinatorial Optimization
  (Enschede, 1991).

\bibitem{West96}
Douglas~B. West.
\newblock {\em Introduction to graph theory}.
\newblock Prentice Hall Inc., Upper Saddle River, NJ, 1996.

\bibitem{Yegnanarayanan02}
V.~Yegnanarayanan.
\newblock On a question concerning prime distance graphs.
\newblock {\em Discrete Math.}, 245(1-3):293--298, 2002.

\end{thebibliography}
\end{document}